\documentclass[a4paper]{amsart}

\usepackage{amsmath, amsfonts, amsthm , amssymb, amscd}
\usepackage{graphicx}
\usepackage{pst-grad} % For gradients and colours
\usepackage{url}
\usepackage{pstricks}

\newcommand\str{\rightarrow}
\newcommand\mj{\mbox{\rm id}}
\newcommand\set{\mbox{\emph{Set}}}

\newcommand\cob{\textbf{\emph{2Cob}}}
\newcommand\vek{\textbf{\emph{Vect}}}

 \newtheorem{thm}{Theorem}[section]
 \newtheorem{prop}[thm]{Proposition}
 \newtheorem{lem}[thm]{Lemma}

\numberwithin{equation}{section}

\title{A faithful 2-dimensional TQFT}

\author{Stevan Gajovi\' c, Zoran Petri\' c and Sonja Telebakovi\' c Oni\' c}
\address{ \scriptsize{University of Groningen, Bernoulli Institute\\
Nijenborgh 9\\ 9747 AG, Groningen, the Netherlands}}
\email{s.gajovic@rug.nl}
\address{ \scriptsize{Mathematical Institute SANU\\ Knez Mihailova 36, p.f.\ 367\\
        11001 Belgrade, Serbia}}
\email{zpetric@mi.sanu.ac.rs}
\address{\scriptsize{University of Belgrade, Faculty of Mathematics\\ Studentski trg 16\\ 11000 Belgrade, Serbia}}
\email{sonjat@matf.bg.ac.rs}

\date{}

\begin{document}

\begin{abstract}
It has been shown in this paper that the commutative Frobenius algebra
$\mathbb{QZ}_5\otimes Z(\mathbb{QS}_3)$ provides a complete
invariant for two-dimensional cobordisms, i.e.,\ that the corresponding
two-dimensional quantum field theory is faithful. Zsigmondy's Theorem is essential to the proof of this result.

\vspace{.3cm}


\vspace{.5ex}

\noindent {\small {\it Keywords$\,$}: Frobenius algebra,
topological quantum field theory, faithful functor, Zsigmondy's
Theorem}
\end{abstract}

\maketitle

\section{Introduction}

It is evident that one aspect of topological quantum field theories (TQFTs) deals with the corresponding invariants of manifolds. However, the completeness of these invariants is seldom investigated in the literature. The existence of a faithful $n$-dimensional TQFT provides a complete algebraic invariant for $n$-cobordisms, which may help in their classification.

In the current article we show that there is a faithful 2-dimensional TQFT. As the classification of 2-cobordisms is well-known, this result has a stronger algebraic than topological impact. It shows that there is a commutative Frobenius algebra, which satisfies only the equalities in the language of multiplication, unit, comultiplication and counit, which hold in every commutative Frobenius object.

Since this structure is free of additional equations, one could be tempted to call it ``free commutative Frobenius algebra''. However, since the category of commutative Frobenius algebras is a groupoid---every homomorphism is an isomorphism (cf.\ \cite[Lemma~2.4.5]{K03}---there are no freely generated objects in this category.

Let $\mathcal{F}$ be the category whose objects are symmetric monoidal categories with one distinguished commutative Frobenius object and whose arrows are symmetric monoidal functors preserving distinguished objects and their Frobenius structures. The forgetful functor from $\mathcal{F}$ to the category $\set$ of sets and functions, which maps a symmetric monoidal category to the set of its objects, has a left adjoint, the ``free'' functor $F$. Let us denote the category $F\emptyset$ by $\mathbf{K}$. This category may be constructed such that its set of objects is the set of finite ordinals and that 1 is the universal commutative Frobenius object, in the same sense as the object 1 of the simplicial category $\Delta$ is the universal monoid. (A detailed construction of $\mathbf{K}$ is given in \cite[Appendix~9.3]{BPT}.)

The standard presentation of the category of 2-cobordisms by generators and relations (cf.\ \cite[Section~1.4]{K03}) may be treated as a completeness result (for a logician) or a coherence result (for a category theorist) of the syntax represented by the category $\mathbf{K}$ with respect to the semantics represented by the category of 2-cobordisms. The existence of a commutative Frobenius algebra, which is free of additional equations in the canonical language, provides a new completeness result with the universal object 1 interpreted not as a topological object, but as an algebraic object.

Our result shows that the category of vector spaces contains an isomorphic copy of the category $\mathbf{K}$. Without proving the existence of a faithful 2-dimensional TQFT, one cannot be sure that commutative Frobenius objects in the category of vector spaces do not satisfy some additional equations, and, hence, deserve some attributes attached to the standard name ``commutative Frobenius algebra''. The algebra constructed in this paper justifies this standard name.

There is a result in \cite[Section~14]{DP12} that claims faithfulness of a 1-dimensional TQFT, which is inspired by \cite{B37}. The third
author, in her recent work \cite{T17}, has shown that every
1-dimensional TQFT, over a field of characteristic zero, is faithful. This means that every such 1-dimensional TQFT provides a complete invariant for 1-cobordisms. On the other hand, we do not know whether there exists a faithful $n$-dimensional TQFT for $n\geq 3$. An important step towards a solution of this problem was given in \cite{J14}, where the author presents the cobordism category in arbitrary dimension $n$ with generators and relations. Our proof for 2-dimensional case suggests that for $n\geq 3$, particular difficulties could be caused by closed manifolds with many connected components. As it was shown in \cite{F13}, neither Turaev-Viro, \cite{TV92}, nor Reshetikhin-Turaev, \cite{T94}, 3-dimensional TQFTs are faithful. We are aware of the fact that even a negative answer to this question might be conclusive---it suggests that TQFTs should find more appropriate targets than the category of vector spaces. Such an approach is provided by extended topological field theories (cf.\ \cite{L09}).

In order to keep this paper as short as possible, we rely on
\cite{K03} for basic definitions, as well as on now classical works \cite{A88}, \cite{Q95}, \cite{A96}, and more recent \cite{L09}.

\section{The category \cob\ and 2TQFTs}

Let \cob\ be the category whose objects are
$\mathbf{0},\mathbf{1},\mathbf{2},\ldots$, where $\mathbf{n}$ is
the sequence of $n$ circles and whose arrows are the equivalence
classes of 2-cobordisms defined as in \cite[Section~1.2]{K03}. We
denote cobordisms by $K,L,\ldots$, and $K=L$ means that $K$ and
$L$ belong to the same equivalence class. In the illustrations below, cobordisms are oriented \emph{top to bottom} (the ingoing boundary is at the top and the outgoing boundary is at the bottom of the picture), not \emph{left to right} as usual.

Let $K\colon\mathbf{n}\str \mathbf{m}$ be a 2-cobordism whose ingoing
and outgoing boundaries are respectively the sequences of circles
$(\Sigma^0_0,\ldots,\Sigma^{n-1}_0)$ and
$(\Sigma^0_1,\ldots,\Sigma^{m-1}_1)$. We define an equivalence
relation $\rho_K$ on the set
\[
(\{0,\ldots,n-1\}\times\{0\})\cup(\{0,\ldots,m-1\}\times\{1\})
\]
such that $(i,k)\rho_K(j,l)$ when $\Sigma^i_k$ and $\Sigma^j_l$
belong to the same connected component in $K$ (cf.\
\cite[Section~8]{BPT}). For example, if $K\colon \mathbf{3}\str
\mathbf{4}$ is
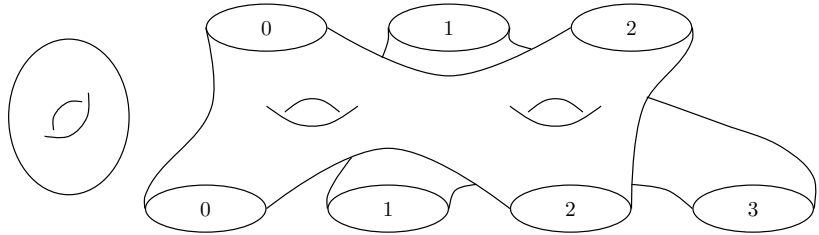
\begin{center}
\scalebox{.8} % Change this value to rescale the drawing.
{
\begin{pspicture}(12.5,4.8)(0,.2)
%\psframe[linewidth=0.04,dimen=outer](3.58,-1.38)(1.34,-2.3)
%\psline[linewidth=0.04cm](1.28,-2.2)(1.28,-1.42)
%\usefont{T1}{ptm}{m}{n}
%\rput(2.3045313,-1.805){\footnotesize Illustration of the first implication}
\psellipse[linewidth=0.02,dimen=outer](4,4)(1,.4)
\psellipse[linewidth=0.02,dimen=outer](7,4)(1,.4)
\psellipse[linewidth=0.02,dimen=outer](10,4)(1,.4)
\psellipse[linewidth=0.02,dimen=outer](3,1)(1,.4)
\psellipse[linewidth=0.02,dimen=outer](6,1)(1,.4)
\psellipse[linewidth=0.02,dimen=outer](9,1)(1,.4)
\psellipse[linewidth=0.02,dimen=outer](12,1)(1,.4)

\psellipse[linewidth=0.02,dimen=outer](0.75,2.52)(1,1.3)
\pscurve[linewidth=0.02,dimen=outer](.35,2.2)(.75,2.2)(1.05,2.52)(1.07,2.92)
\pscurve[linewidth=0.02,dimen=outer](.5,2.3)(.5,2.52)(.75,2.77)(.97,2.77)

\rput(4,4){0} \rput(7,4){1} \rput(10,4){$2$} \rput(3,1){0}
\rput(6,1){1} \rput(9,1){$2$} \rput(12,1){$3$}

\pscurve[linewidth=0.02,dimen=outer](3,4)(3.1,2.7)(2,1.2)(2,1)
\pscurve[linewidth=0.02,dimen=outer](11,4)(11,3.8)(10.2,2.7)(10,1)
\pscurve[linewidth=0.02,dimen=outer](5,4)(7,3.2)(9,4)
\pscurve[linewidth=0.02,dimen=outer](4,1)(6,2)(8,1)
\pscurve[linewidth=0.02,dimen=outer](6,4)(6,3.8)(5.9,3.5)
\pscurve[linewidth=0.02,dimen=outer](5,1)(5,1.2)(5.4,1.85)
\pscurve[linewidth=0.02,dimen=outer](7,1)(7.1,1.3)(7.45,1.4)
\pscurve[linewidth=0.02,dimen=outer](10,1.4)(10.6,1.3)(11,1)
\pscurve[linewidth=0.02,dimen=outer](8,4)(8,3.8)(8.35,3.65)
\pscurve[linewidth=0.02,dimen=outer](10.25,2.85)(11.5,2.4)(12.5,2)(13,1.5)(13,1)

\pscurve[linewidth=0.02,dimen=outer](4,2.7)(4.5,2.4)(5,2.4)(5.5,2.7)
\pscurve[linewidth=0.02,dimen=outer](4.3,2.6)(4.6,2.8)(4.9,2.8)(5.2,2.6)
\pscurve[linewidth=0.02,dimen=outer](8,2.7)(8.5,2.4)(9,2.4)(9.5,2.7)
\pscurve[linewidth=0.02,dimen=outer](8.3,2.6)(8.6,2.8)(8.9,2.8)(9.2,2.6)

\end{pspicture}
}
\end{center}
then the equivalence classes of $\rho_K$ are
\[
\{(0,0),(2,0),(0,1),(2,1)\}\quad{\rm and}\quad
\{(1,0),(1,1),(3,1)\}.
\]
Also, we denote by $(g^i_k)_K$ the genus of the connected
component of $K$ containing~$\Sigma^i_k$.

The category \cob\ is a symmetric monoidal with the \emph{tensor
product} $\otimes$ given by ``putting side by side'' and symmetry
generated by the transpositions:
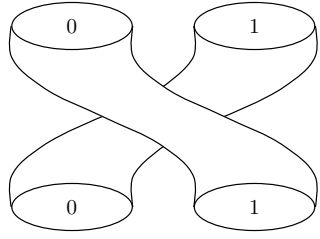
\begin{center}
\scalebox{.8} % Change this value to rescale the drawing.
{
\begin{pspicture}(3,4.5)(0,.5)
\psellipse[linewidth=0.02,dimen=outer](0,4)(1,.4)
\psellipse[linewidth=0.02,dimen=outer](3,4)(1,.4)
\psellipse[linewidth=0.02,dimen=outer](0,1)(1,.4)
\psellipse[linewidth=0.02,dimen=outer](3,1)(1,.4)

\rput(0,4){0} \rput(3,4){1} \rput(0,1){0}
\rput(3,1){1}

\pscurve[linewidth=0.02,dimen=outer](-1,4)(-1,3.5)(-.5,3)(.5,2.5)(1.5,2)(2,1.5) (2,1)
\pscurve[linewidth=0.02,dimen=outer](1,4)(1,3.5)(1.5,3)(2.5,2.5)(3.5,2)(4,1.5) (4,1)
\pscurve[linewidth=0.02,dimen=outer](2,4)(2,3.5)(1.5,3)
\pscurve[linewidth=0.02,dimen=outer](.5,2.5)(-.5,2)(-1,1.5)(-1,1)
\pscurve[linewidth=0.02,dimen=outer](4,4)(4,3.5)(3.5,3)(2.5,2.5)
\pscurve[linewidth=0.02,dimen=outer](1.5,2)(1,1.5)(1,1)

\end{pspicture}
}
\end{center}

Let \vek\ be the category of vector spaces over a fixed field
whose symmetric monoidal structure is given by the tensor product and
the usual symmetry. According to Atiyah's axioms (see
\cite[Section~2]{A88}), a 2-\emph{dimensional quantum field
theory} (2TQFT) is a symmetric, strong monoidal functor (cf.\
\cite[Section XI.2]{ML71}) from \cob\ to \vek.

For $m,k,n\geq 0$, let $E_{m,k,n}$ denote the connected
2-cobordism with $n$ ingoing boundaries, $m$ outgoing boundaries
and genus $k$.
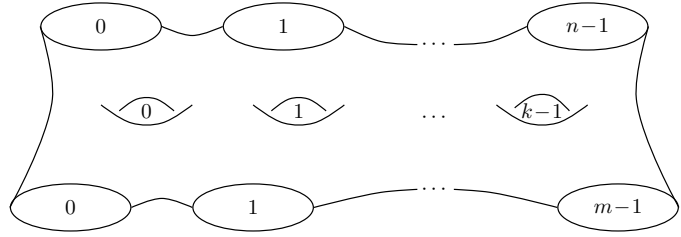
\begin{center}
\scalebox{.8} % Change this value to rescale the drawing.
{
\begin{pspicture}(8,4.8)(0,.2)
\psellipse[linewidth=0.02,dimen=outer](0,4)(1,.4)
\psellipse[linewidth=0.02,dimen=outer](3,4)(1,.4)
\psellipse[linewidth=0.02,dimen=outer](8,4)(1,.4)
\psellipse[linewidth=0.02,dimen=outer](-.5,1)(1,.4)
\psellipse[linewidth=0.02,dimen=outer](2.5,1)(1,.4)
\psellipse[linewidth=0.02,dimen=outer](8.5,1)(1,.4)

\rput(0,4){0} \rput(3,4){1} \rput(8,4){$n\!-\!1$} \rput(-.5,1){0}
\rput(2.5,1){1} \rput(8.5,1){$m\!-\!1$} \rput(5.5,3.7){$\cdots$}
\rput(5.5,1.3){$\cdots$} \rput(5.5,2.5){$\cdots$}
\rput(0.75,2.6){0} \rput(3.25,2.6){1} \rput(7.25,2.6){$k\!-\!1$}

\pscurve[linewidth=0.02,dimen=outer](-1,4)(-.8,2.8)(-1.5,1)
\pscurve[linewidth=0.02,dimen=outer](9,4)(8.8,2.8)(9.5,1)
\pscurve[linewidth=0.02,dimen=outer](1,4)(1.5,3.85)(2,4)
\pscurve[linewidth=0.02,dimen=outer](.5,1)(1,1.15)(1.5,1)
\pscurve[linewidth=0.02,dimen=outer](4,4)(4.6,3.75)(5.2,3.7)
\pscurve[linewidth=0.02,dimen=outer](7,4)(6.4,3.75)(5.8,3.7)
\pscurve[linewidth=0.02,dimen=outer](3.5,1)(4.6,1.25)(5.2,1.3)
\pscurve[linewidth=0.02,dimen=outer](7.5,1)(6.4,1.25)(5.8,1.3)

\pscurve[linewidth=0.02,dimen=outer](0,2.7)(.5,2.4)(1,2.4)(1.5,2.7)
\pscurve[linewidth=0.02,dimen=outer](0.3,2.6)(.6,2.8)(.9,2.8)(1.2,2.6)
\pscurve[linewidth=0.02,dimen=outer](2.5,2.7)(3,2.4)(3.5,2.4)(4,2.7)
\pscurve[linewidth=0.02,dimen=outer](2.8,2.6)(3.1,2.8)(3.4,2.8)(3.7,2.6)
\pscurve[linewidth=0.02,dimen=outer](6.5,2.7)(7,2.4)(7.5,2.4)(8,2.7)
\pscurve[linewidth=0.02,dimen=outer](6.8,2.65)(7.1,2.85)(7.4,2.85)(7.7,2.6)

\end{pspicture}
}
\end{center}

As a part of a relation between 2TQFTs and commutative Frobenius algebras, which is thoroughly explained in \cite[Section~3.3]{K03}, we have that if $F$ is a 2TQFT, then for
\[
\mu=F(E_{1,0,2}),\quad \eta=F(E_{1,0,0}),\quad\delta=F(E_{2,0,1}),\quad {\rm and}\quad \varepsilon=F(E_{0,0,1}),
\]
$(F\mathbf{1},\mu,\eta,\delta,\varepsilon)$, is a commutative Frobenius algebra. Conversely, if $(A,\mu,\eta,\delta,\varepsilon)$ is a commutative Frobenius algebra, then there is a 2TQFT, which we denote by $F_A$, mapping $\mathbf{1}$ into $A$, and $E_{1,0,2}$, $E_{1,0,0}$, $E_{2,0,1}$ and $E_{0,0,1}$ into $\mu$, $\eta$, $\delta$ and $\varepsilon$, respectively. For such an $F_A$, we denote $F_AK$ by $(K)_A$, and abbreviate $F_AK=F_AL$ by $K=_A L$.

The following three lemmata hold since 2TQFT is a monoidal functor.

\begin{lem}[FILLING HOLES]\label{FH}
If $K=_A L$ for $K,L\colon \mathbf{n}\str \mathbf{m}$, then for every $0\leq i\leq n-1$ and $0\leq j\leq m-1$, we have
\[
K\circ(\mj_i\otimes E_{1,0,0}\otimes\mj_{n-i-1})=_A L\circ(\mj_i\otimes E_{1,0,0}\otimes\mj_{n-i-1})
\]
and
\[
(\mj_j\otimes E_{0,0,1}\otimes\mj_{m-j-1})\circ K=_A (\mj_j\otimes
E_{0,0,1}\otimes\mj_{m-j-1})\circ L.
\]
\end{lem}

\begin{center}
\scalebox{.8} % Change this value to rescale the drawing.
{
\begin{pspicture}(13,3.8)(0,-0.4)
\rput(6.5,-0.2){\small Illustration of the first implication of Lemma~\ref{FH}}
\psellipse[linewidth=0.02,dimen=outer](0,3.2)(.5,.2)
\psellipse[linewidth=0.02,dimen=outer](2,3.2)(.5,.2)
\psellipse[linewidth=0.02,dimen=outer](1,.8)(.5,.2)

\psarc[linewidth=0.02,dimen=outer](2,3.2){.5}{180}{360}

\pscurve[linewidth=0.02,dimen=outer](-.2,2.7)(-.05,2.5)(.15,2.5)(.3,2.7)
\pscurve[linewidth=0.02,dimen=outer](-.12,2.64)(-.05,2.7)(.15,2.7)(.22,2.64)

\pscurve[linewidth=0.02,dimen=outer](-.5,3.2)(-.4,2.5)(0,2)(.4,1.5)(.5,.8)
\pscurve[linewidth=0.02,dimen=outer](.5,3.2)(.6,2.5)(1,2)(1.4,1.5)(1.5,.8)

\rput(3,2){\huge $=_A$}

\psellipse[linewidth=0.02,dimen=outer](4,3.2)(.5,.2)
\psellipse[linewidth=0.02,dimen=outer](6,3.2)(.5,.2)
\psellipse[linewidth=0.02,dimen=outer](5,.8)(.5,.2)

\psarc[linewidth=0.02,dimen=outer](5,3.2){.5}{180}{360}
\psarc[linewidth=0.02,dimen=outer](5,3.2){1.5}{180}{360}
\psarc[linewidth=0.02,dimen=outer](5,.8){.5}{0}{180}

\rput(7.5,2){\huge $\Rightarrow$}

\psellipse[linewidth=0.02,dimen=outer](9,3.2)(.5,.2)
\psellipse[linewidth=0.02,dimen=outer](9,.8)(.5,.2)
\psellipse[linewidth=0.02,dimen=outer](10.5,2)(.5,.2)

\psline[linewidth=0.02](8.5,3.2)(8.5,.8)
\psline[linewidth=0.02](9.5,3.2)(9.5,.8)

\pscircle[linewidth=0.02,dimen=outer](10.5,2){.5}

\pscurve[linewidth=0.02,dimen=outer](8.8,2.7)(8.95,2.5)(9.15,2.5)(9.3,2.7)
\pscurve[linewidth=0.02,dimen=outer](8.88,2.64)(8.95,2.7)(9.15,2.7)(9.22,2.64)

\rput(12,2){\huge $=_A$}

\psellipse[linewidth=0.02,dimen=outer](13,3.2)(.5,.2)
\psellipse[linewidth=0.02,dimen=outer](13,.8)(.5,.2)

\psarc[linewidth=0.02,dimen=outer](13,3.2){.5}{180}{360}
\psarc[linewidth=0.02,dimen=outer](13,.8){.5}{0}{180}

\end{pspicture}
}
\end{center}

\begin{lem}[STRETCHING 1]\label{S1}
If $K=_A L$ for $K,L\colon \mathbf{1}\str \mathbf{0}$, then we have
\[
(K\otimes \mj_1)\circ E_{2,0,1}=_A (L\otimes \mj_1)\circ E_{2,0,1},
\]
and if $K=_A L$ for $K,L\colon \mathbf{0}\str \mathbf{1}$, then we have
\[
E_{1,0,2}\circ (K\otimes \mj_1)=_A E_{1,0,2}\circ (L\otimes
\mj_1).
\]
\end{lem}

\begin{center}
\scalebox{.8} % Change this value to rescale the drawing.
{
\begin{pspicture}(13,3.4)(0,.0)
\rput(6.5,.2){\small Illustration of the first implication of Lemma~\ref{S1}}
\psellipse[linewidth=0.02,dimen=outer](0,2)(.5,.75)
\psellipse[linewidth=0.02,dimen=outer](1.7,3)(.5,.2)

\psarc[linewidth=0.02,dimen=outer](1.7,2.15){1}{120}{420}

\pscurve[linewidth=0.02,dimen=outer](-.25,2.1)(-.1,1.9)(.1,1.9)(.25,2.1)
\pscurve[linewidth=0.02,dimen=outer](-.17,2.04)(-.1,2.1)(.1,2.1)(.17,2.04)

\pscurve[linewidth=0.02,dimen=outer](1.45,2.1)(1.6,1.9)(1.8,1.9)(1.95,2.1)
\pscurve[linewidth=0.02,dimen=outer](1.53,2.04)(1.6,2.1)(1.8,2.1)(1.87,2.04)

\rput(3.5,2){\huge $=_A$}

\psellipse[linewidth=0.02,dimen=outer](5.2,3)(.5,.2)

\psarc[linewidth=0.02,dimen=outer](5.2,2.15){1}{120}{420}

\pscurve[linewidth=0.02,dimen=outer](4.55,2.1)(4.7,1.9)(4.9,1.9)(5.05,2.1)
\pscurve[linewidth=0.02,dimen=outer](4.63,2.04)(4.7,2.1)(4.9,2.1)(4.97,2.04)

\pscurve[linewidth=0.02,dimen=outer](5.35,2.1)(5.5,1.9)(5.7,1.9)(5.85,2.1)
\pscurve[linewidth=0.02,dimen=outer](5.43,2.04)(5.5,2.1)(5.7,2.1)(5.77,2.04)

\rput(7.5,2){\huge $\Rightarrow$}

\psellipse[linewidth=0.02,dimen=outer](9,2)(.5,.75)
\psellipse[linewidth=0.02,dimen=outer](10.7,3)(.5,.2)
\psellipse[linewidth=0.02,dimen=outer](10.7,1)(.5,.2)

\pscurve[linewidth=0.02,dimen=outer](8.75,2.1)(8.9,1.9)(9.1,1.9)(9.25,2.1)
\pscurve[linewidth=0.02,dimen=outer](8.83,2.04)(8.9,2.1)(9.1,2.1)(9.17,2.04)

\pscurve[linewidth=0.02,dimen=outer](10.45,2.1)(10.6,1.9)(10.8,1.9)(10.95,2.1)
\pscurve[linewidth=0.02,dimen=outer](10.53,2.04)(10.6,2.1)(10.8,2.1)(10.87,2.04)

\psline[linewidth=0.02](10.2,3)(10.2,1)
\psline[linewidth=0.02](11.2,3)(11.2,1)

\rput(12,2){\huge $=_A$}

\psellipse[linewidth=0.02,dimen=outer](13,3)(.5,.2)
\psellipse[linewidth=0.02,dimen=outer](13,1)(.5,.2)

\psline[linewidth=0.02](12.5,3)(12.5,1)
\psline[linewidth=0.02](13.5,3)(13.5,1)

\pscurve[linewidth=0.02,dimen=outer](12.75,2.6)(12.9,2.4)(13.1,2.4)(13.25,2.6)
\pscurve[linewidth=0.02,dimen=outer](12.83,2.54)(12.9,2.6)(13.1,2.6)(13.17,2.54)

\pscurve[linewidth=0.02,dimen=outer](12.75,1.6)(12.9,1.4)(13.1,1.4)(13.25,1.6)
\pscurve[linewidth=0.02,dimen=outer](12.83,1.54)(12.9,1.6)(13.1,1.6)(13.17,1.54)

\end{pspicture}
}
\end{center}

\begin{lem}[STRETCHING 2]\label{S2}
If $K=_A L$ for $K,L\colon \mathbf{2}\str \mathbf{0}$, then we have
\[
(K\otimes \mj_1)\circ (\mj_1\otimes E_{2,0,0})=_A (L\otimes
\mj_1)\circ (\mj_1\otimes E_{2,0,0}),
\]
and if $K=_A L$ for $K,L\colon \mathbf{0}\str \mathbf{2}$, then we have
\[
(\mj_1\otimes E_{0,0,2})\circ (K\otimes \mj_1)=_A (\mj_1\otimes
E_{0,0,2})\circ (L\otimes \mj_1).
\]
\end{lem}

\begin{center}
\scalebox{.8} % Change this value to rescale the drawing.
{
\begin{pspicture}(13,3.4)(0,.0)
\rput(6.5,.2){\small Illustration of the first implication of Lemma~\ref{S2}}
\psellipse[linewidth=0.02,dimen=outer](0,3)(.5,.2)
\psellipse[linewidth=0.02,dimen=outer](2,3)(.5,.2)

\psarc[linewidth=0.02,dimen=outer](1,3){.5}{180}{360}
\psarc[linewidth=0.02,dimen=outer](1,3){1.5}{180}{360}

\pscurve[linewidth=0.02,dimen=outer](.75,2.1)(.9,1.9)(1.1,1.9)(1.25,2.1)
\pscurve[linewidth=0.02,dimen=outer](.83,2.04)(.9,2.1)(1.1,2.1)(1.17,2.04)

\rput(3.5,2){\huge $=_A$}

\psellipse[linewidth=0.02,dimen=outer](4.5,3)(.5,.2)
\psellipse[linewidth=0.02,dimen=outer](6,3)(.5,.2)

\psarc[linewidth=0.02,dimen=outer](4.5,3){.5}{180}{360}
\psarc[linewidth=0.02,dimen=outer](6,3){.5}{180}{360}

\rput(7.5,2){\huge $\Rightarrow$}

\psellipse[linewidth=0.02,dimen=outer](10,3)(.5,.2)
\psellipse[linewidth=0.02,dimen=outer](10,1)(.5,.2)

\psline[linewidth=0.02](9.5,3)(9.5,1)
\psline[linewidth=0.02](10.5,3)(10.5,1)

\pscurve[linewidth=0.02,dimen=outer](9.75,2.4)(9.9,2.2)(10.1,2.2)(10.25,2.4)
\pscurve[linewidth=0.02,dimen=outer](9.83,2.34)(9.9,2.4)(10.1,2.4)(10.17,2.34)

\rput(11.7,2){\huge $=_A$}

\psellipse[linewidth=0.02,dimen=outer](13,3.2)(.5,.2)
\psellipse[linewidth=0.02,dimen=outer](13,.8)(.5,.2)

\psarc[linewidth=0.02,dimen=outer](13,3.2){.5}{180}{360}
\psarc[linewidth=0.02,dimen=outer](13,.8){.5}{0}{180}

\end{pspicture}
}
\end{center}

\begin{prop}[MAXIMALITY]\label{max}
If for $K\neq L$, we have $K=_A L$, where $\dim(A)>1$, then
for some $k_1\geq\ldots\geq k_n\geq 0$ and
$l_1\geq\ldots\geq l_m\geq 0$ such that
$(k_1,\ldots,k_n)\neq(l_1,\ldots,l_m)$, we have
\begin{equation}\label{eq0}
\bigotimes_{i=1}^n E_{0,k_i,0}=_A \bigotimes_{j=1}^m E_{0,l_j,0}.
\end{equation}
\end{prop}

\begin{proof}
Since $\dim(A)>1$, the cobordisms $K$ and $L$ must have the same
source and target. Also, $K\neq L$ entails that either $\rho_K\neq \rho_L$, or $\rho_K=\rho_L$ and there is $(i,k)$ such that
$(g^i_k)_K\neq (g^i_k)_L$, or $\rho_K=\rho_L$ and for every $(i,k)$,
$(g^i_k)_K= (g^i_k)_L$ and $K$ and $L$ differ in their closed components.

We start with the last and simplest case. If $\rho_K=\rho_L$ and for every $(i,k)$ we have
$(g^i_k)_K= (g^i_k)_L$, then by applying Lemma~\ref{FH} for all
the boundary components, we arrive at the equality of the form (\ref{eq0}).

If $\rho_K=\rho_L$ and there is $(i,k)$ such that
$(g^i_k)_K\neq (g^i_k)_L$, then by applying Lemma~\ref{FH} for all
the boundary components except the one corresponding to $(i,k)$,
and then by applying Lemma~\ref{S1}, we arrive at the equality of
the form
\begin{equation}\label{eq2}
E_{1,p,1}\otimes (\bigotimes_{i=1}^n E_{0,k_i,0})=_A
E_{1,q,1}\otimes (\bigotimes_{j=1}^m E_{0,l_j,0}),
\end{equation}
for some $n,m,p,q\geq 0$ such that $p\neq q$, and $k_1\geq\ldots\geq k_n\geq 0$,
$l_1\geq\ldots\geq l_m\geq 0$.

If $\rho_K\neq \rho_L$ and $(i,k)\rho_K(j,l)$, while not
$(i,k)\rho_L(j,l)$, then by applying Lemma~\ref{FH} for all the
boundary components except those corresponding to $(i,k)$ and
$(j,l)$ we arrive either directly at the equality of the form
\begin{equation}\label{eq1}
E_{1,p,1}\otimes (\bigotimes_{i=1}^n E_{0,k_i,0}) =_A
E_{1,q,0}\otimes
E_{0,r,1}\otimes (\bigotimes_{j=1}^m E_{0,l_j,0}),
\end{equation}
for some $n,m,p,q,r\geq 0$ and $k_1\geq\ldots\geq k_n\geq 0$,
$l_1\geq\ldots\geq l_m\geq 0$, or this equality is obtained by a further application of
Lemma~\ref{S2}.

For $a>\max\{k_1,l_1\}$, put the both sides of the equalities (\ref{eq2}) and (\ref{eq1}) in the context $E_{0,a,1}\circ\mbox{\underline{\hspace{1em}}}\circ E_{1,a,0}$ in order to obtain the equality of the form (\ref{eq0}).
\end{proof}

\section{Frobenius algebras $\mathbb{QZ}_5$ and $Z(\mathbb{QS}_3)$}

For all the examples below, when we fix a basis $\langle
\beta_1,\ldots,\beta_n\rangle$ of a vector space $V$, then we
assume that the tensor product $V\otimes V$ has the fixed basis
\[
\langle
\beta_1\otimes\beta_1,\beta_1\otimes\beta_2,\ldots,\beta_2\otimes\beta_1,\ldots,\beta_n\otimes\beta_n\rangle,
\]
and we represent the linear transformations by matrices with
respect to these bases.

For $\mathbb{Z}_5$ being the cyclic group of order 5, with the
generator $a$, let $\mathbb{QZ}_5$ be the group algebra and let
$\langle e,a,a^2,a^3,a^4\rangle$ be its basis. The multiplication
$\mu\colon \mathbb{QZ}_5\otimes \mathbb{QZ}_5\str \mathbb{QZ}_5$ is
represented by the $5\times 25$ matrix $M$ {\scriptsize\[ \left[
\begin{array}{ccccccccccccccccccccccccc}
1 &0 &0 &0 &0 &0 &0 &0 &0 &1 &0 &0 &0 &1 &0 &0 &0 &1 &0 &0 &0 &1
&0 &0 &0
\\
0 &1 &0 &0 &0 &1 &0 &0 &0 &0 &0 &0 &0 &0 &1 &0 &0 &0 &1 &0 &0 &0
&1 &0 &0
\\
0 &0 &1 &0 &0& 0 &1 &0 &0 &0 &1 &0 &0 &0 &0 &0 &0 &0 &0 &1 &0 &0
&0 &1 &0
\\
0 &0 &0 &1 &0& 0 &0 &1 &0 &0& 0 &1 &0 &0 &0 &1 &0 &0 &0 &0 &0 &0
&0 &0 &1
\\
0 &0 &0 &0 &1& 0 &0 &0 &1 &0& 0 &0 &1 &0 &0& 0 &1 &0 &0 &0 &1 &0
&0 &0 &0
\end{array}
\right]
\]}
while the unit $\eta\colon \mathbb{Q}\str\mathbb{QZ}_5$ is represented
by the $5\times 1$ matrix:
\[ \left[
\begin{array}{c}
1
\\
0
\\
0
\\
0
\\
0
\end{array}
\right]
\]
The comultiplication $\delta\colon \mathbb{QZ}_5\str
\mathbb{QZ}_5\otimes \mathbb{QZ}_5$ is represented by the
$25\times 5$ matrix ${1\over 5}M^T$, and the counit
$\varepsilon\colon \mathbb{QZ}_5\str \mathbb{Q}$ is represented by the
$1\times 5$ matrix
\[ \left[
\begin{array}{ccccc}
5 & 0 & 0 & 0 & 0
\end{array}
\right].
\]

The structure $(\mathbb{QZ}_5,\mu,\eta,\delta,\varepsilon)$ is a
commutative Frobenius algebra and it is \emph{special} in the
sense that for every $k$
\[
E_{1,k,1}=_{\mathbb{QZ}_5} E_{1,0,1}.
\]
Note that $(E_{0,k,0})_{\mathbb{QZ}_5}$ is represented by the $1\times 1$ matrix, i.e.\ the rational number 5.

Let $Z(\mathbb{QS}_3)$ be the center of the group algebra
$\mathbb{QS}_3$, where $\mathbb{S}_3$ is the symmetric group of degree~3. Denote the three conjugacy classes of
$\mathbb{S}_3$ by $C_1=\{e\}$, $C_2=\{(12),(13),(23)\}$
and $C_3=\{(123),(132)\}$. By \cite[Proposition~12.22]{JL01} one can fix
\[
\langle e, (12)+(13)+(23), (123)+(132)\rangle
\]
as the basis of $Z(\mathbb{QS}_3)$.

The multiplication $\mu\colon Z(\mathbb{QS}_3)\otimes Z(\mathbb{QS}_3)\str Z(\mathbb{QS}_3)$ is represented by the $3\times 9$ matrix
\[ \left[
\begin{array}{ccccccccc}
1 & 0 & 0 & 0 & 3 & 0 & 0 & 0 & 2
\\
0 & 1 & 0 & 1 & 0 & 2 & 0 & 2 & 0
\\
0 & 0 & 1 & 0 & 3 & 0 & 1 & 0 & 1
\end{array}
\right],
\]
the unit $\eta\colon \mathbb{Q}\str Z(\mathbb{QS}_3)$ is represented by the $3\times 1$ matrix
\[ \left[
\begin{array}{c}
1
\\
0
\\
0
\end{array}
\right],
\]
while the Frobenius form, i.e.\ the counit $\varepsilon\colon  Z(\mathbb{QS}_3)\str \mathbb{Q}$ is represented by the $1\times 3$ matrix
\[ \left[
\begin{array}{ccc}
1 & 0 & 0
\end{array}
\right].
\]

Since the Frobenius pairing $\beta$ is equal to $\varepsilon\circ\mu$ it is represented by the matrix
\[ \left[
\begin{array}{ccccccccc}
1 & 0 & 0 & 0 & 3 & 0 & 0 & 0 & 2
\end{array}
\right],
\]
hence, the corresponding copairing $\gamma$ is represented by the transpose of the matrix
\[ \left[
\begin{array}{ccccccccc}
1 & 0 & 0 & 0 & \frac{1}{3} & 0 & 0 & 0 & \frac{1}{2}
\end{array}
\right].
\]
The comultiplication in a Frobenius algebra is given by the composition \[
(\mj\otimes\mu)\circ(\gamma\otimes\mj)=(\mu\otimes\mj)\circ(\mj\otimes\gamma),
\]
and the coassociativity of the comultiplication is a consequence of the associativity of the multiplication. Therefore, the comultiplication $\delta\colon  Z(\mathbb{QS}_3)\str  Z(\mathbb{QS}_3)\otimes  Z(\mathbb{QS}_3)$ is represented by the $9\times 3$ matrix, which is the transpose of
\[ \left[
\begin{array}{ccccccccc}
1 & 0 & 0 & 0 & \frac{1}{3} & 0 & 0 & 0 & \frac{1}{2}
\\
0 & 1 & 0 & 1 & 0 & 1 & 0 & 1 & 0
\\
0 & 0 & 1 & 0 & \frac{2}{3} & 0 & 1 & 0 & \frac{1}{2}
\end{array}
\right].
\]

It is easy to check that for the commutative Frobenius algebra $(Z(\mathbb{QS}_3),\mu,\eta,\delta,\varepsilon)$ we have that $(E_{1,1,1})_{Z(\mathbb{QS}_3)}$, which is equal to $\mu\circ\delta$, is represented by the matrix
\[ \left[
\begin{array}{ccc}
3 & 0 & 3
\\
0 & 6 & 0
\\
\frac{3}{2} & 0 & \frac{9}{2}
\end{array}
\right].
\]
Hence, $(E_{1,k,1})_{Z(\mathbb{QS}_3)}$ is represented by the $k$th power of this matrix, i.e.\ by the matrix
\[ \left(\frac{3}{2}\right)^{k-1}\left[
\begin{array}{ccc}
2^{2k-1}+1 & 0 & 2^{2k}-1
\\
0 & 3\cdot 2^{2k-1} & 0
\\
2^{2k-1}-\frac{1}{2} & 0 & 2^{2k}+\frac{1}{2}
\end{array}
\right].
\]
Eventually,  $(E_{0,k,0})_{Z(\mathbb{QS}_3)}=\varepsilon\circ (E_{1,k,1})_{Z(\mathbb{QS}_3)}\circ\eta$ is represented by the
rational number
\[
\left(\frac{3}{2}\right)^{k-1}\!(2^{2k-1}+1).
\]

\section{Faithfulness}

In this section we denote the tensor product $\mathbb{QZ}_5\otimes Z(\mathbb{QS}_3)$ by $\mathbb{A}$. The algebra $\mathbb{A}$ is equipped with the commutative Frobenius structure as the tensor product of two such algebras (cf.\ \cite[Section~2.4]{K03}). Note that $(E_{0,k,0})_{\mathbb{A}}$ is represented by the rational number
\[
5\cdot\left(\frac{3}{2}\right)^{k-1}\!(2^{2k-1}+1).
\]

The following theorem, known as \emph{Zsigmondy's Theorem for sums} \cite{Zsi} (see also \cite[P1.7]{R94} and \cite{R97}), and the subsequent lemma are crucial for the proof of the faithfulness of the 2TQFT corresponding to $\mathbb{A}$.

\begin{thm}
For positive integers $a$, $b$ and $n$ such that $a$, $b$ are coprime, $a>b$ and $(n,a,b)\neq(3,2,1)$, there is a prime number $p$ such that $p$ divides $a^n + b^n$ and for every $k<n$, $p$ does not divide $a^k + b^k$.
\end{thm}

\begin{lem}\label{l4}
If for $k_1\geq\ldots\geq k_n\geq 0$ and $l_1\geq\ldots\geq l_m\geq 0$
\[
\prod_{i=1}^n\left(5\cdot \left(\frac{3}{2}\right)^{k_i-1}\!(2^{2k_i-1}+1)\right)=
\prod_{j=1}^m\left(5\cdot \left(\frac{3}{2}\right)^{l_j-1}\!(2^{2l_j-1}+1)\right),
\]
then $n=m$ and $(k_1,\ldots,k_n)=(l_1,\ldots,l_m)$.
\end{lem}

\begin{proof}
Let $p$ and $q$ be such that $k_p,l_q>0$ and $k_{p+1}=0=l_{q+1}$ (if there are any). Then the above equality reads
\[
5^n\cdot\frac{3^{(\sum_{i=1}^p k_i)-p}}{2^{(\sum_{i=1}^p k_i)-p}}\prod_{i=1}^p(2^{2k_i-1}+1)= 5^m\cdot\frac{3^{(\sum_{j=1}^q l_j)-q}}{2^{(\sum_{j=1}^q l_j)-q}}\prod_{j=1}^q(2^{2l_j-1}+1).
\]
Since the last digit in $2^{2k-1}+1$ is either 3 or 9, such a factor is not divisible by 5, and we may conclude that $n=m$. Since all the factors but $2^{(\sum_{i=1}^p k_i)-p}$ and $2^{(\sum_{j=1}^q l_j)-q}$ are odd, we may conclude that $(\sum_{i=1}^p k_i)-p=(\sum_{j=1}^q l_j)-q$, and that
\[
\prod_{i=1}^p(2^{2k_i-1}+1)=\prod_{j=1}^q(2^{2l_j-1}+1).
\]

If $(k_1,\ldots,k_p)\neq(l_1,\ldots,l_q)$, then, after
cancelation, we may assume that every $k_i$ is different from
every $l_j$. Assume also that $k_1>l_1$. It is not possible that
$k_1=2$, since then $l_1=\ldots=l_q=1$ and $(\sum_{j=1}^q
l_j)-q=0<(\sum_{i=1}^p k_i)-p$. Hence, $k_1\geq 3$ and $2k_1-1\geq 5$.
By applying Zsigmondy's Theorem for sums, there would be a prime that
divides $2^{2k_1-1}+1$ and for every $1\leq j\leq q$ it does not
divide $2^{2l_j-1}+1$, which contradicts the above equality.
\end{proof}

\begin{thm}\label{t1}
The \emph{2TQFT} $F_\mathbb{A}$ is faithful and injective on objects.
\end{thm}

\begin{proof}
Since $\dim(\mathbb{A})>1$, the functor $F_\mathbb{A}$ is injective on objects. By Lemma~\ref{l4}, for every $k_1\geq \ldots\geq k_n\geq 0$ and $l_1\geq \ldots\geq l_m\geq 0$ such that $(k_1,\ldots,k_n)\neq(l_1,\ldots,l_m)$ we have
\[
\bigotimes_{i=1}^n E_{0,k_i,0}\neq_\mathbb{A} \bigotimes_{j=1}^m E_{0,l_j,0},
\]
and it remains for Proposition~\ref{max} to be applied.
\end{proof}

\begin{center}\textmd{\textbf{Acknowledgements} }
\end{center}
\medskip
We would like to thank the anonymous referees for suggestions, which helped to improve the paper, and for a comment how to simplify the formulation of Proposition~\ref{max} and the proof of Theorem~\ref{t1}.
We also thank Djordje Barali\' c for some very useful suggestions. This work was supported by projects 174026 and 174032 of the Ministry of Education, Science, and Technological Development of the Republic of Serbia. The first author was supported by DFG-Grant MU 4110/1-1.

\end{document}